\documentclass[11pt,final]{amsart}
\usepackage[active]{srcltx}
\usepackage[latin1]{inputenc}
\usepackage{amsmath,bbm}
\usepackage{amsfonts}
\usepackage{pgf,tikz}
\usetikzlibrary{arrows}
\usepackage{amssymb}
\usepackage[colorlinks=true,urlcolor=blue,
citecolor=red,linkcolor=blue,linktocpage,pdfpagelabels,bookmarksnumbered,bookmarksopen]
{hyperref}
\usepackage{graphicx, tikz}
\usepackage{pgf,tikz}
\usetikzlibrary{arrows}
\usepackage[active]{srcltx}
\usepackage[latin1]{inputenc}
\usepackage{amsmath,bbm}
\usepackage{amsfonts}
\usepackage{amsthm}
\usepackage{amssymb}
\usepackage{amsfonts}
\usepackage{hyperref}
\usepackage{hyperref}
\usepackage{amsmath,accents,eqnarray}
\usepackage[english]{babel}
\usepackage{marginnote}
\usepackage[left=2.95cm,right=2.95cm,top=2.8cm,bottom=2.8cm]{geometry}
\numberwithin{equation}{section}
\usepackage{mathrsfs}
\usepackage{color}
\usepackage{amsmath, bbm}
\usepackage{amsfonts}
\usepackage{amssymb}
\usepackage{graphicx}%
\usepackage[active]{srcltx}
\usepackage[latin1]{inputenc}
\usepackage{amsmath,bbm}
\usepackage{amsfonts}
\usepackage{pgf,tikz}
\usetikzlibrary{arrows}
\usepackage{amssymb}
\newtheorem{theorem}{Theorem}[section]
\newtheorem{proposition}[theorem]{Proposition}
\newtheorem{lemma}[theorem]{Lemma}

\newtheorem{definition}{Definition}[section]
\newtheorem{remark}{Remark}[section]

\numberwithin{equation}{section}

\numberwithin{equation}{section}
\usepackage{mathrsfs}

\DeclareMathOperator{\spt}{spt}
\DeclareMathOperator{\Lip}{Lip}

\begin{document}
\title[Summability estimates on transport densities]
{Optimal transportation with boundary costs and summability estimates on the transport density }
\author[S. Dweik]{Samer Dweik}
\address{Laboratoire de Math\'ematiques d'Orsay, Univ. Paris-Sud, CNRS, Universit\'e Paris-Saclay, 91405 Orsay Cedex, France}
\email{Samer.Dweik@math.u-psud.fr}
\maketitle
\begin{abstract} In this paper we analyze a mass transportation problem in a bound\-ed
domain with the possibility to transport mass to/from the boundary, paying the transport cost, that is given by the Euclidean distance plus an extra cost depending on the exit/entrance point. This problem appears in import/export model, as well as in some shape optimization problems.
We study the $L^p$ summability of the transport density which does not follow from standard theorems, as the target measures are not absolutely continuous but they have some parts of them which are concentrated on the boundary. We also provide the relevant duality arguments to connect the corresponding Beckmann and Kantorovich problems to a formulation with Kantorovich potentials with Dirichlet boundary conditions.
\end{abstract}

 
  \section{Introduction} \label{1}
  In this paper we study a mass transportation problem in a bounded
domain where there is the possibility of import/export mass across the boundary
paying a tax fee in addition to the transport cost that is assumed to be given by the Euclidean distance. Before entering the details of this variant problem, let us introduce the standard Kantorovich problem. \\

Let $f^+$ and $f^-$ be two finite positive measures on a bounded domain $\Omega \subset \mathbb{R}^d$ satisfying the mass balance condition $f^+(\bar{\Omega})=f^-(\bar{\Omega})$. The classical Kantorovich problem is the following\,:
\\

Set $$\Pi(f^+,f^-):=\left\{\gamma \in \mathcal{M}^+(\bar{\Omega} \times \bar{\Omega})\,:\,(\Pi_{x})_{ \#}\gamma =f^+,\,(\Pi_{y})_{ \#}\gamma =f^-\;\right\},$$
then we minimize the quantity
$$\min\left\{\int_{\bar{\Omega} \times \bar{\Omega} }|x-y|\mathrm{d}\gamma:\,\gamma \in \Pi(f^+,f^-)\right\}\qquad(\mbox{KP})$$
where $\Pi_{x}$ and $\Pi_y$ are the two projections of $\bar{\Omega} \times \bar{\Omega}$ onto $\bar{\Omega}$.\\

In \cite{7}, the authors introduce a variant of (KP). They study a mass transportation problem between two masses $f^+$ and $f^-$ (which do not have a priori the same total masses) with the possibility of transporting some mass to/from the boundary, paying the transport cost $c(x,y):=|x-y|$ plus an extra cost $g_2(y)$ for each unit of mass that comes out from a point $y \in \partial\Omega$ (the export taxes) or $-g_1(x)$ for each unit of mass that enters at the point $x \in \partial\Omega$ (the import taxes). This means that we can use $\partial\Omega$ as an infinite reserve/repository, we can take as much mass as we wish from the boundary, or send back as much mass as we want, provided that we pay the transportation cost plus the import/export taxes. \\

In other words, given the set
$$\Pi b(f^+,f^-):=\left\{\gamma \in \mathcal{M}^+(\bar{\Omega}\times\bar{\Omega})\,:\,(\Pi_{x})_{ \#}(\gamma) .1_{\accentset{\circ}{\Omega}}=f^+,\,(\Pi_{y})_{ \#}(\gamma) .1_{\accentset{\circ}{\Omega}}=f^-\right\},$$
 we minimize the quantity
$$
\min\left\{\int_{\bar{\Omega}\times \bar{\Omega}}|x-y|\mathrm{d}\gamma
  +\int_{\partial\Omega}g_2\mathrm{d}(\Pi_{y})_{\#}\gamma -\int_{\partial\Omega}g_1\mathrm{d}(\Pi_{x})_{\#}\gamma
   \,:\,\gamma \in \Pi b (f^+,f^-) \right\}.\qquad(\mbox{KPb})$$
\\
On the other hand, let us consider the following problem 
$$\min\left\{|W|(\bar{\Omega})\,:\,W \in \mathcal{M}^d(\bar{\Omega}),\,\nabla\cdot W=f \right\} \quad \mbox{(BP)}$$
where $\mathcal{M}^d(\bar{\Omega})$ is the space of vector measures and for $W \in \mathcal{M}^d(\bar{\Omega})$, $|W|(\bar{\Omega})$ denotes the total variation measure (note that $|W|(\bar{\Omega})$ is a norm on $\mathcal{M}^d(\bar{\Omega})$).\\

It is well known that if $\Omega$ is convex, then (BP) is equivalent to (KP), i.e the values of both problems are equal and we can construct a minimizer for (BP) from a minimizer for (KP) and vice versa. 
\\
\\
 From the equality $\min \mbox{(KP)}=\min \mbox{(BP)}$, it is easy to see that $\min \mbox{(KPb)}=\min \mbox{(BPb)}$, where (BPb) is the following problem
$$\min\left\{|W|(\bar{\Omega})+\int_{\partial\Omega}g_2\mathrm{d}\chi^- -\int_{\partial\Omega}g_1\mathrm{d}\chi^+\,:\,W \in \mathcal{M}^d(\bar{\Omega}),\,\chi \in \mathcal{M}(\partial\Omega),\,\nabla\cdot W=f+\chi \right\}.$$ \\
Before building a minimizer $W$ for (BP), take a minimizer $\gamma$ for (KP) (which is called \textit{optimal transport plan}) and define the transport density $\sigma$ associated with $\gamma$ as follows
\begin{equation} \label{DefTrans}
<\sigma,\varphi>=\int_{\bar{\Omega} \times \bar{\Omega}}\mathrm{d}\gamma(x,y)\int_{0}^{1}\varphi(\omega_{x,y}(t))|\dot{\omega}_{x,y}(t)|\mathrm{d}t\;\;\;\mbox{for all}\;\;\varphi \in C(\bar{\Omega})
\end{equation}
where $\,\omega_{x,y}$ is a curve parametrizing the straight line segment connecting $x$ to $y$. \\ 

Then it is easy to check that the vector field $W$ given by $W=-\sigma \nabla u$ is a solution of the above minimization problem (BP), where $u$ is a maximizer (called \textit{Kantorovich potential}) for the following problem 
$$\sup\left\{ \int_{\Omega} u\,\mathrm{d}(f^+ -f^-)\,:\,u\in \Lip_1 \right\}.\qquad \mbox{(DP)}$$
Actually, it is possible to prove that the maximization problem above is the dual of (KP) and its value equals $\min \mbox{(KP)}$.
\\ 

In addition, 
 $(\sigma,u)$ solves a particular PDE system, called Monge-Kantorovich system\,:
\begin{equation}\label{MKsyst1}
\begin{cases}
-\nabla\cdot (\sigma\nabla u)=f=f^+-f^-&\mbox{ in }\Omega\\
\sigma\nabla u\cdot n=0 &\mbox{ on }\partial\Omega\\
|\nabla u|\leq 1&\mbox{ in }\Omega,\\
|\nabla u|= 1&\sigma-\mbox{a.e. }\end{cases}
\end{equation}
The summability of $\sigma$ has been the object of intensive research in the last few years, and in particular we have the following\,:
\begin{proposition}\label{prop transp dens}
 Suppose $f^+ \ll \mathcal{L}^d$ or $f^- \ll \mathcal{L}^d$, then the transport density $\sigma$ is unique (i.e. does not depend on the choice of the optimal transport plan $\gamma$) and $\sigma \ll \mathcal{L}^d$.
Moreover, if both $ f^+,\,f^- \in L^p(\Omega)$, then $\sigma$ also belongs to $L^p(\Omega)$.\end{proposition}
These properties are well-known in the literature, and we refer to \cite{55}, \cite{66},  \cite{77}, \cite{33} and \cite{9}.\\ 

Hence, if $\,f^+ \ll \mathcal{L}^d\,$ or $\,f^- \ll \mathcal{L}^d\,$, then (BP) is also well-posed in $L^1$ instead of in the space of vector measure. In addition, $W:=- \sigma \nabla u$ (which minimizes (BP)) belongs to $L^p(\Omega,\mathbb{R}^d)$ provided that $f \in L^p(\Omega)$. \\
 
Suppose that $f$ has not zero mass, then a variant of this problem, which is already present in  \cite{BouButJEMS},\,\cite{4} and \cite{DweSan}, is to complete the Monge-Kantorovich system with a Dirichlet boundary condition. In optimal transport terms, this corresponds to the possibility of transporting some mass to/from the boundary, paying only the transport cost that is given by the Euclidean distance. The easiest version of the system becomes 
\begin{equation}\label{MKsyst2}
\begin{cases}
-\nabla\cdot (\sigma\nabla u)=f &\mbox{ in }\Omega\\
u=0 &\mbox{on }\partial\Omega,\\
|\nabla u|\leq 1&\mbox{ in }\Omega,\\
|\nabla u|= 1&\sigma-\mbox{a.e. }\end{cases}
\end{equation}
Notice that in \cite{DumIgb,PieCan}, the same pair $(\sigma,u)$ (which solves (\ref{MKsyst2})) also models (in a statical or dynamical framework) the configuration of stable or growing sandpiles, where $u$ gives the pile shape and $\sigma$ stands for sliding layer. 
\\

But, we can replace also the Dirichlet boundary condition $u=0$ by $u=g$. In this case, the system becomes 
\begin{equation}\label{MKsyst22}
\begin{cases}
-\nabla\cdot (\sigma\nabla u)=f &\mbox{ in }\Omega\\
u=g &\mbox{on }\partial\Omega,\\
|\nabla u|\leq 1&\mbox{ in }\Omega,\\
|\nabla u|= 1&\sigma-\mbox{a.e. }\end{cases}
\end{equation}
This system describe the growth of a sandpile on a bounded table, with a wall on the boundary of a height $g$, under the
action of a vertical source here modeled by $f$. Notice that to solve this system, it is clear that $g$ must be 1-Lipschitz.
\\

In \cite{7}, the system (\ref{MKsyst22}) is complemented with boundary conditions $g_1 \leq u \leq g_2$ , instead of $u=g$, on $\partial\Omega$. The system becomes 
\begin{equation}\label{MKsyst222}
\begin{cases}
-\nabla\cdot (\sigma\nabla u)=f &\mbox{ in }\Omega\\
g_1 \leq u \leq g_2 &\mbox{on }\partial\Omega,\\
|\nabla u|\leq 1&\mbox{ in }\Omega,\\
|\nabla u|= 1&\sigma-\mbox{a.e. }\end{cases}
\end{equation}
Here, there is no obvious interpretation in terms of sandpiles. However, this corresponds to a mass transportation problem between two masses $f^+$ and $f^-$ with the possibility of transporting some mass to/from the boundary, paying a transport cost plus two extra costs 
  $-g_1$ and $g_2$ (the import/export taxes). \\
     
   The authors of \cite{7} also prove that the dual of (KPb) is the following 
$$\sup\left\{ \int_{\Omega} u\,\mathrm{d}(f^+ -f^-)\,:\,u\in \Lip_1,\; g_1 \leq u \leq g_2\,\mbox{ on }\partial\Omega\right\}.\qquad \mbox{(DPb)}$$
  The reader will see later that in this paper we also give an alternative proof for this duality formula that we consider simpler than that in \cite{7}.\\
 
 In addition, if we are able to prove that the vector measure $W:=- \sigma \nabla u$ minimizes (BPb), where $u$ is a maximizer for (DPb) and $\sigma$ is the transport density associated with an optimal transport plan $\gamma$ for (KPb), then (BPb) is well-posed in $L^1(\Omega,\mathbb{R}^d)$ instead of the space of vector measure as soon as one has $\sigma \ll \mathcal{L}^d$. Here, we need the convexity of $\Omega$ to define $\sigma$ (see \ref{DefTrans}), but we will show that under some assumptions on $g_1$ and $g_2$, we can also use \ref{DefTrans} to define $\sigma$, even if $\Omega$ is not convex.\\
 
   The main object of the present paper is the problem (KPb). First, we give an alternative proof for its dual formulation, which is already proved in \cite{7} and second, we are interested to study the $L^p$ summability of the transport density $\sigma$, which does not follow from Proposition \ref{prop transp dens}, since in this case
   the target measures are not in $L^p$ as they have some parts which are concentrated on $\partial\Omega$. Note that in \cite{DweSan}, the authors prove that if $g_1=g_2=0$, then the transport density $\sigma$ belongs to $L^p$ provided that $f \in L^p$. Here, our goal is to prove the same $L^p$ result of \cite{DweSan} but for more general costs $g_1$ and $g_2$. First of all, we note that to get a $L^p$ summability on $\sigma$, it is natural to suppose that $g_i$ is strictly better than 1-Lip. Indeed, we can find $f \in L^p(\Omega)$ and $\chi \in \mathcal{M}(\partial\Omega)$ such that the transport density $\sigma$ between $f^+ + \chi^+$ and $f^- + \chi^-$ is not in $L^p(\Omega)$ and in this case if $u$ is the Kantorovich potential (which is 1-Lip), then $(\sigma,u)$ solves (\ref{MKsyst222}) with $g_1=g_2=u$.\\
   
 For this aim, we want to decompose an optimal transport plan $\gamma$ for (KPb) to a sum of three transport plans $\gamma_{ii},\,\gamma_{ib}$ and $\gamma_{bi}$, where each of these plans solves a particular transport problem. Next, we will study the $L^p$ summability of the transport densities $\sigma_{ii},\,\sigma_{ib}$ and $\sigma_{bi}$ associated with these transport plans $\gamma_{ii},\,\gamma_{ib}$ and $\gamma_{bi}$, respectively. In this way, we get the summability of the transport density $\sigma$ associated with the optimal transport plan $\gamma$.\\

This paper is organized as follows. In Section \ref{8}, we prove that it is enough to study the summability of $\sigma_{ib}$, to get that of $\sigma$
and we study duality for (KPb). In Section \ref{2}, we study the $L^p$ summability of the transport density $\sigma_{ib}\,$: firstly, we prove it under an assumption on the geometric form of $\,\Omega\,$ and secondly, we generalize the result to every domain having a uniform exterior ball.
 In Section \ref{4}, we prove direclty the $L^p$ summability of the transport density $\sigma_{ib}$, only for the case $g_2=0$, by using a geometric lemma. In Section \ref{5}, we  give the proofs of the key Propositions already used in Section \ref{2}, which are very technical and we found better to postpone their presentation. 
   
   \section{Monge-Kantorovich problems with boundary costs: existence, characterization and duality} \label{8}
 In this section, we 
 analyze the problem (KPb). Besides duality questions, we will also decompose it into subproblems. One of this subproblems involves a transport plan $\gamma_{ib}$ (with its transport density $\sigma_{ib})$, where $i$ and $b$ stand for interior and boundary (conversely, we also have a transport plan $\gamma_{bi}$ with $\sigma_{bi}$). We will show that some questions, inclusing summability of $\sigma$, reduce to the study of the summability of $\sigma_{ib}$ and $\sigma_{bi}$.\\
 
  First of all, we suppose that $\,g_1$ and $\,g_2$ are in $C(\partial\Omega)$ and they satisfy the following inequality
 \begin{equation} \label{g_1g_2}
 g_1(x)-g_2(y)\leq |x-y|\;\mbox{for all}\;x,\,y \in \partial\Omega.
 \end{equation}
 
Under this assumption, we have the following\,:
\begin{proposition} \label{existence of a minimizer}(KPb) reaches a minimum.
   \end{proposition}
  \begin{proof} 
 
  First, we observe that if $\gamma \in \Pi b (f^{+},f^{-})$ and  $\tilde{\gamma}:=\gamma .1_{(\partial\Omega\times\partial\Omega)^c}$, then $\tilde{\gamma}$ also belongs to $\Pi b (f^{+},f^{-})$. In addition, we have
$$\int_{\bar{\Omega}\times\bar{\Omega}} |x\;-\;y|\mathrm{d} \gamma +\int_{\partial\Omega}g_2\mathrm{d}(\Pi_{y})_{\#}\gamma - \int_{\partial\Omega}g_1\mathrm{d}(\Pi_{x})_{\#}\gamma$$
$$=\int_{\partial{\Omega}\times\partial{\Omega}} (|x-y| + g_2(y) - g_1(x))\mathrm{d} \gamma + \int_{(\partial{\Omega}\times\partial{\Omega})^c}|x-y|\mathrm{d} \gamma + \int_{{\Omega}^{\circ}\times\partial{\Omega}} g_2(y)\mathrm{d} \gamma 
- \int_{\partial{\Omega}\times{\Omega}^{\circ}} g_1(x)\mathrm{d} \gamma.$$ \\
 As $$|x-y| + g_2(y) - g_1(x) \geq 0 ,$$
  we get 
  \begin{multline*}
  \int_{\bar{\Omega}\times\bar{\Omega}} |x-y|\mathrm{d} \gamma\; +\int_{\partial\Omega}g_2\mathrm{d}(\Pi_{y})_{\#}\gamma\, - \int_{\partial\Omega}g_1\mathrm{d}(\Pi_{x})_{\#}\gamma \\ \geq  \int_{\bar{\Omega}\times\bar{\Omega}} |x-y|\mathrm{d}\tilde{\gamma} \,+\int_{\partial\Omega}g_2\mathrm{d}(\Pi_{y})_{\#}\tilde{\gamma} \,- \int_{\partial\Omega}g_1\mathrm{d}(\Pi_{x})_{\#}\tilde{\gamma}.
  \end{multline*}
  \\
Now, let $(\gamma_n)_n\subset\Pi b(f^+,f^-)$ be a minimizing sequence. Then, we can suppose that $$\gamma_n (\partial\Omega\times\partial\Omega)=0.$$ 
In this case, we get 
\begin{eqnarray*}
\gamma_n(\bar{\Omega}\times\bar{\Omega}) &\leq  & \gamma_n({\Omega}^0\times\bar{\Omega})+
  \gamma_n(\bar{\Omega}\times{\Omega}^0)
\\
&=& f^+(\bar{\Omega}) + f^-(\bar{\Omega}).
\end{eqnarray*}
 Hence, there exist a subsequence $ (\gamma_{n_k})_{n_k}$ and a plan $\gamma \in \Pi b (f^+,f^-) $ such that
 $ \gamma_{n_k} {\rightharpoonup}\,\gamma$. On the other hand, if 
 $$K(\gamma):=\int_{\bar{\Omega}\times\bar{\Omega}} |x-y|\mathrm{d} \gamma\; +\int_{\partial\Omega}g_2\mathrm{d}(\Pi_{y})_{\#}\gamma\, - \int_{\partial\Omega}g_1\mathrm{d}(\Pi_{x})_{\#}\gamma,$$
 then it is easy to see that $K$ is continuous with respect to the weak convergence of measures in $\Pi b(f^+,f^-)$ and $\gamma$ solves (KPb).
 $\qedhere$
  \end{proof}
 Fix a minimizer $\gamma$ for (KPb) and denote by $\chi^+$ and $\chi^-$ the two positive measures concentrated on the boundary of $\Omega$ such that $(\Pi_{x})_{\#}\gamma=f^++\chi^+  $ and $ (\Pi_{y})_{\#}\gamma=f^-+\chi^-.$ Then, we may see that $\gamma$ is also a minimizer for the following problem
$$
\min\left\{\int_{\bar{\Omega}\times\bar{\Omega}}|x-y|
 \mathrm{d}\gamma\,:\,\gamma \in \Pi(\mu^+,\mu^-)\right\} $$
 where $\mu^\pm:=f^\pm+\chi^\pm$.
\\ 

Set $$\gamma_{ii}:=\gamma. 1_{\Omega^{\circ}\times\Omega^{\circ}},\,
\gamma_{ib}:=\gamma. 1_{\Omega^{\circ}\times\partial\Omega},\,
\gamma_{bi}:=\gamma. 1_{\partial\Omega\times\Omega^{\circ}},\,
\gamma_{bb}:=\gamma. 1_{\partial\Omega\times\partial\Omega}=0 $$ and
$$\nu^+:=(\Pi_{x})_{\#}(\gamma_{ib}),
 \,\nu^-:=(\Pi_{y})_{\#}(\gamma_{bi}).$$ \\
Consider the three following problems\,:

$$
\min \left\{\int_{\bar{\Omega}\times\bar{\Omega}} |x-y|\mathrm{d}\gamma
\,:\,\gamma \in \Pi  (f^{+}-\nu^+,f^{-}-\nu^-)\right\}\;\; \mbox{(P1)}$$
$$
\min \left\{\int_{\bar{\Omega}\times\bar{\Omega}} |x-y|\mathrm{d}\gamma+\int_{\partial{\Omega}}g_2\mathrm{d}(\Pi_{y})_{\#}\gamma
\,:\,\gamma \in  \mathcal{M}^+(\bar{\Omega}\times\bar{\Omega})
,\,(\Pi_{x})_{\#}\gamma=\nu^+ ,\,\spt((\Pi_{y})_{\#}\gamma) \subset \partial\Omega \right\}  \;\;\mbox{(P2)}$$
$$
 \min \left\{\int_{\bar{\Omega}\times\bar{\Omega}} |x-y|\mathrm{d}\gamma -
\int_{\partial{\Omega}}g_1\mathrm{d}(\Pi_{x})_{\#}\gamma
\,:\,\gamma \in \mathcal{M}^+(\bar{\Omega}\times\bar{\Omega})
,\,(\Pi_{y})_{\#}\gamma =\nu^- ,\, \spt((\Pi_{x})_{\#}\gamma) \subset \partial\Omega \right\}. \;\;\mbox{(P3)}$$ \\
 
It is not difficult to prove that $\gamma_{ii},\,\gamma_{ib}$ and $\gamma_{bi}$ solve (P1),\,(P2) and (P3), respectively. In addition, we can see that 
$\gamma_{ib}\,$ is of the form $(Id,T_{ib})_{\#} \nu^+$ and that it solves the following problem
$$\min\left\{\int_{\bar{\Omega}\times\bar{\Omega}}|x-y|\mathrm{d}\gamma
 \,:\,\gamma \in \Pi(\nu^+,(T_{ib})_{\#}\nu^+) \right\},$$
  where $\,T_{ib}(x):=\mbox{argmin}\left\{|x-y|+g_2(y),\,y \in \partial\Omega\right\} \; \mbox{for all} \;x \in \bar{\Omega}.$
  \\
 
  In the same way, $\gamma_{bi}\,$ is of the form $(T_{bi},Id)_{\#} \nu^-$ and it also solves 
 $$\min\left\{\int_{\bar{\Omega}\times\bar{\Omega}}|x-y|\mathrm{d}\gamma
 \,:\,\gamma \in \Pi((T_{bi})_{\#}\nu^-,\nu^-) \right\},$$
where $\,T_{bi}(y):=\mbox{argmin}\left\{|x-y|-g_1(x),\,x \in \partial\Omega\right\} \; \mbox{for all} \;y \in \bar{\Omega}.$
 \\

Let $\sigma$ (resp. $\sigma_{ii},\,\sigma_{ib}$ and $\sigma_{bi}$) be the transport density associated with the optimal transport plan $\gamma$ (resp. $\gamma_{ii},\,\gamma_{ib}$  and $\gamma_{bi}$), therefore $\sigma=\sigma_{ii} + \sigma_{ib} + \sigma_{bi}$. By Proposition \ref{prop transp dens}, if $\Omega$ is convex and $f \in L^p(\Omega)$, then $\sigma_{ii}$ also belongs to $L^p(\Omega)$. Hence, it is enough to study the summability of $\sigma_{ib}$ (the case of $\sigma_{bi}$ will be analogous) to get that of $\sigma$. \\

On the other hand, the proof of the duality formula of (KPb), in \cite{7}, is based on the Fenchel-Rocafellar duality Theorem and it is decomposed into two steps\,: firstly, the authors suppose that the inequality in (\ref{g_1g_2}) is strict and secondly, they use an approximation argument to cover the other case. 
  Here, we want to give an alternative proof for this duality formula, based on a simple convex analysis trick already developed in \cite{BouButJEMS}.\\

\begin{proposition}\label{duality} Let $g_1$ and $g_2$ be in $C(\partial\Omega)
 $. 
Then under the assumption (\ref{g_1g_2}), we have the following equality \\
$$\min \left\{\int_{\bar{\Omega}\times\bar{\Omega}}|x-y|\mathrm{d}\gamma+
\int_{\partial\Omega}g_2\mathrm{d}(\Pi_{y})_{\#}\gamma-\int_{\partial\Omega}g_1\mathrm{d}(\Pi_{x})_{\#}\gamma\,:\,\gamma \in \Pi b(f^+,f^-)\right\} \quad \mbox{(KPb)}$$
$$ =\sup\left\{\int_\Omega \varphi \mathrm{d}(f^+ - f^-)
\,:\,\varphi\in\Lip_1,\,g_1\leq\varphi\leq g_2\; \mbox{on}\;\,
\partial\Omega\right\}.\quad \mbox{(DPb)}$$
\\
Notice that if (\ref{g_1g_2}) is not satisfied, then both sides of this equality are $-\infty$. \\
\end{proposition}
\begin{proof} 
 For every $ (p,q) \in  C(\partial\Omega)\times C(\partial\Omega),$ set $$ H(p,q):=-\sup\bigg\{\int_\Omega \varphi \mathrm{d}(f^+ - f^-)
\,:\,\varphi\in \Lip_1,\,g_1+p\leq\varphi \leq g_2-q \;\mbox{on}\;\partial\Omega\bigg\}.$$ 
It is easy to see that $H(p,q) \in \mathbb{R} \cup \{+\infty\}$. In addition, we claim that $H$ is convex and l.s.c.
 \\

   For convexity\,: take $t \in (0,1)$ and $(p_0,q_0),\,(p_1,q_1)\in C(\partial\Omega)\times C(\partial\Omega) $
 and let $\varphi_0,\,\varphi_1$ be their optimal potentials. Set $$ p_t=(1-t)p_0+tp_1,\,q_t=(1-t)q_0+tq_1 $$ 
 and $$\varphi_t=(1-t)\varphi_0+t\varphi_1. $$
As $$ g_1+p_0 \leq \varphi_0 \leq g_2 - q_0 \;\;\mbox{and}\;\;g_1+p_1 \leq \varphi_1 \leq g_2 - q_1\;\mbox{on}\; \partial\Omega, $$ 
then $$ g_1+p_t \leq \varphi_t \leq g_2 - q_t \;\mbox{on}\; \partial\Omega.$$ \\
In addition, $
 \varphi_t$ is 1-Lip. Consequently, $\varphi_t$ is admissible in the max defining $-H(p_t,q_t)$ and 
 $$ H(p_t,q_t)\leq -\int_\Omega\varphi_t \mathrm{d}(f^+ -f^-)
 =(1-t)H(p_0,q_0)+tH(p_1,q_1). $$
 \\
 
 For semi-continuity$\,$: take $p_n\rightarrow p$ and $q_n \rightarrow q$  uniformly on $\partial\Omega$. Let $(p_{n_k},q_{n_k})_{n_k}$ be a subsequence such that $\liminf_n H(p_n,q_n)=\lim_{n_k}H(p_{n_k},q_{n_k})$ (for simplicity of notation, we still denote this subsequence by $(p_n,q_n)_n$) and let $(\varphi_n)_n$ be
  their corresponding optimal potentials. As $\varphi_n$ is 1-Lip and $(p_n)_n,\,(q_n)_n$ are equibounded, then, by Ascoli-Arzel\`a Theorem, there exist a 1-Lip function $\varphi$ and a subsequence $(\varphi_{n_k})_{n_k}$ such that $
 \varphi_{n_k}\rightarrow\varphi$ uniformly.
   \\
 \\
 As $$ g_1+p_{n_k} \leq \varphi_{n_k} \leq g_2 - q_{n_k} \;\mbox{on}\; \partial\Omega, $$
 then $$ g_1+p \leq \varphi \leq g_2 - q \;\mbox{on}\; \partial\Omega.
 $$ \\
 Consequently, $\varphi$ is admissible in the max defining $ -H(p,q)$ and one has
 $$ H(p,q)\leq -\int_\Omega\varphi \mathrm{d}(f^+ - f^-)
 =\lim_{n_k} H(p_{n_k},q_{n_k})=\liminf_n H(p_n,q_n). $$ \\
 Hence, we get that $ H^{\star\star}=H $ and in particular, $H^{\star \star}(0,0)=H(0,0)$. But by the definition of $H$, we have  $$H(0,0)=-\sup\left\{\int_\Omega \varphi \mathrm{d}(f^+ - f^-)
\,:\,\varphi \in \Lip_1,\,g_1 \leq \varphi \leq g_2\; \mbox{on}\;
\partial\Omega\right\}.$$ 
\\
On the other hand, let us compute $H^{\star\star}(0,0)$. Take $\chi^\pm$ in $\mathcal{M}(\partial\Omega)$, then we have 
$$H^{\star}(\chi^+,\chi^-) =
  \sup_{p,\,q \,\in \,C(\partial\Omega)
}\bigg\{\int_{\partial\Omega} p\mathrm{d}\chi^+ + \int_{\partial\Omega} q\mathrm{d}\chi^--H(p,q)\bigg\}$$
$$=\sup \bigg\{\int_{\partial\Omega} p\mathrm{d}\chi^+ + \int_{\partial\Omega} q\mathrm{d}\chi^- +\int_\Omega \varphi\mathrm{d}(f^+ -f^-)\,:\,p,\,q\in C(\partial\Omega),\,\varphi\in\Lip_1,\\
\,g_1 + p\leq\varphi\leq g_2 - q\; \mbox{on}\;
\partial\Omega\bigg\}.$$ \\
If $\chi^+\notin \mathcal{M}^+(\partial\Omega) $, i.e there exists $p_0 \in C(\partial\Omega)$ 
 such that $p_0\geq 0
$ 
and
$ \int_{\partial\Omega} p_0\mathrm{d}\chi^+ <0, 
$ we may see that 
$$H^{\star}(\chi^+,\chi^-)\geq\;-n\int_{\partial\Omega} p_0\mathrm{d}\chi^+ +\int_{\partial\Omega} g_2\mathrm{d}\chi^- -\int_{\partial\Omega} g_1\mathrm{d}\chi^+\;\underset{n \to + \infty}{\longrightarrow}  \;+\infty$$
and similarly if $\chi^-\notin \mathcal{M}^+(\partial\Omega).
 $\\
\\ 
Suppose $\chi^\pm \in \mathcal{M}^+(\partial\Omega) $. As $g_1 + p\leq\varphi\leq g_2 - q \;\mbox{on}\;\partial\Omega,$ we should choose the largest possible $p$ and $q$, i.e $\,p(x)=\varphi(x)-g_1(x)\,$ and $\,q(y)=g_2(y) - \varphi(y)$ for all $x,\,y \in\partial\Omega$. Hence, we have
$$ H^{\star}(\chi^+,\chi^-)=\sup\bigg\{ \int_{\bar{\Omega}} \varphi\mathrm{d}\left(f+ \chi\right)
 \,:\,\varphi \in \Lip_1 \bigg\}  + \int_{\partial\Omega} g_2\mathrm{d}\chi^- -\int_{\partial\Omega} g_1\mathrm{d}\chi^+ ,$$
 where $f:=f^+ - f^-$ and $\chi:=\chi^+ - \chi^-$. \\

By Theorem 1.14 in \cite{11}, we get
$$ H^{\star}(\chi^+,\chi^-)=\min\left\{\int_{\bar{\Omega}\times\bar{\Omega}}|x-y|\mathrm{d}\gamma
\,:\,\gamma \in \Pi(f^+ + \chi^+,f^-+\chi^-) \right\}
+ \int_{\partial\Omega} g_2\mathrm{d}\chi^- -\int_{\partial\Omega} g_1\mathrm{d}\chi^+ $$
$$=\min\left\{\int_{\bar{\Omega}\times\bar{\Omega}}|x-y|\mathrm{d}\gamma
+ \int_{\partial\Omega} g_2\mathrm{d}(\Pi_{y})_{\#}\gamma
-\int_{\partial\Omega} g_1\mathrm{d}(\Pi_{x})_{\#}\gamma
\,:\,\gamma \in \Pi(f^+ + \chi^+,f^-+\chi^-) \right\}.$$\\
Finally, we have
$$H^{\star\star}(0,0)=\sup\bigg\{-H^{\star}(\chi^+,\chi^-)\,:\,\chi^+,\,\chi^-\in\mathcal{M}^+(\partial\Omega)\bigg\}$$
$$=-\min \left\{\int_{\bar{\Omega}\times\bar{\Omega}}|x-y|\mathrm{d}\gamma+
\int_{\partial\Omega}g_2\mathrm{d}(\Pi_y)_{\#}\gamma-
\int_{\partial\Omega}g_1\mathrm{d}(\Pi_x)_{\#}\gamma\,:\,\gamma\in\Pi b(f^+,f^-)\right\}. \qedhere$$
\end{proof}
Let $u$ be a maximizer for (DPb). Then, we have the following\,:
\begin{proposition} \label{Maxdual}
The potential $u$ is also a Kantorovich potential for the following problem 
$$ \sup\left\{\int_{\bar{\Omega}} \varphi \mathrm{d}(\mu^+ - \mu^-)
\,:\,\varphi\in\Lip_1\right\} $$
where $\mu^\pm:=f^\pm + \chi^\pm$.
\end{proposition}
\begin{proof}
Let $v$ be a Kantorovich potential for this dual problem.
Then, we have $$\int_\Omega u \mathrm{d}(f^+ - f^-) + \int_{\partial\Omega}g_1\mathrm{d}\chi^+ - \int_{\partial\Omega}g_2\mathrm{d}\chi^-\leq \int_{\bar{\Omega}} u \mathrm{d}(\mu^+ - \mu^-) \leq \int_{\bar{\Omega}} v \mathrm{d}(\mu^+ - \mu^-).$$
By Proposition \ref{duality} and the fact that (see Theorem 1.14 in \cite{11})$$ \sup\left\{\int_{\bar{\Omega}} \varphi \mathrm{d}(\mu^+ - \mu^-)
\,:\,\varphi\in\Lip_1\right\}=\min\left\{\int_{\bar{\Omega}\times\bar{\Omega}}|x-y|
 \mathrm{d}\gamma\,:\,\gamma\in\Pi(\mu^+,\mu^-)\right\},$$ we infer that these inequalities are in fact equalities and $u$ is a Kantorovich potential for this dual problem.$\qedhere$ \\
\end{proof}
Suppose that $\Omega$ is convex and set 
$W:=-\sigma \nabla u,$
where we recall that $\sigma$ is the transport density associated with the optimal transport plan $\gamma$. Then, from Proposition \ref{Maxdual} and the fact that 
$\min \mbox{(BPb)}=\min \mbox{(KPb)},$ we can conclude that $W$ 
and $\chi$ solve together (BPb). 
Moreover, the same result will be true, even if $\Omega$ is not convex, by using the following\,: 
\begin{proposition} \label{support} Suppose that $$|g_1(x)-g_2(y)| \leq |x-y|\; \;\mbox{for all}\,\;(x,y)\in\partial\Omega \times \partial\Omega, $$
 i.e. $g_1=g_2:=g$ and $g$ is 1-Lip. Then there exists a minimizer $\gamma^{\star}$ for (KPb) such that  for all $(x,y) \in \spt(\gamma^{\star})$, we have $[x,y] \subset \bar{\Omega}.$ In addition, if $g$ is $\lambda$-Lip with $ \lambda< 1$, then for any minimizer $\gamma$ of (KPb) and for all $(x,y)\in\spt(\gamma),\;[x,y] \subset \bar{\Omega}.$
 \end{proposition}
 \begin{proof} Let $\gamma$ be a minimizer for (KPb) and set
  $$E:=\{(x,y) \in \bar{\Omega}\times \bar{\Omega},\,[x,y]\subset \bar{\Omega}\},$$ 
   $$h_1:\bar{\Omega}\times \bar{\Omega}\mapsto \bar{\Omega}\times \partial\Omega$$ $$ (x,y) \mapsto (x,{y}^\prime)$$ where ${y}^\prime$ is the first point of intersection between the segment $[x,y]$ and the boundary if $(x,y) \notin E$ and ${y}^\prime =y$  else.\\
   
Also set $$h_2: \bar{\Omega}\times \bar{\Omega} \mapsto \partial\Omega \times \bar{\Omega}$$
 $$ (x,y) \mapsto ({x}^\prime,y)$$ where ${x}^\prime $ is the last point of intersection between the segment $[x,y]$ and the boundary if $(x,y) \notin E$ and ${x}^\prime =x$  else.\\
\\
 Set $$\gamma^{\star}:=\gamma.1_E + (h_1)_{\#}(\gamma.1_{E^c}) + (h_2)_{\#}(\gamma.1_{E^c}).$$\\
 It is clear that $\gamma^{\star} \in \Pi b(f^+,f^-)$. In addition, we have
  $$\int_{\bar{\Omega}\times\bar{\Omega}} |x-y|\mathrm{d}\gamma^{\star}+\int_{\partial{\Omega}}g(y)\mathrm{d}(\Pi_{y})_{\#}\gamma^{\star} -\int_{\partial{\Omega}}g(x)\mathrm{d}(\Pi_{x})_{\#}\gamma^{\star}$$
 
 $$ =\int_{E} |x-y|\mathrm{d}\gamma+\int_{E^c} (|x-{y}^\prime| + |{x}^\prime-y| + g({y}^\prime) - g({x}^\prime)) \mathrm{d}\gamma +\int_{\partial{\Omega}}g(y)\mathrm{d}(\Pi_{y})_{\#}\gamma
 -\int_{\partial{\Omega}}g(x)\mathrm{d}(\Pi_{x})_{\#}\gamma .$$ 
 \\ Yet, 
  $$|x-{y}^\prime| + |{x}^\prime-y| + g({y}^\prime) - g({x}^\prime) \leq |x-{y}^\prime| + |{x}^\prime-y| + |{x}^\prime-{y}^\prime|=|x-y|.$$ 
    \\
Consequently, $\gamma^{\star}$ is a minimizer for (KPb) and for all $(x,y) \in \spt(\gamma^{\star})$, we have $[x,y] \subset \bar{\Omega}.$ The second statement follows directly from the last inequality, which becomes strict.$\qedhere$
  \end{proof}
 \section{$L^p$ summability on the transport density} \label{2}
In this section, we will study the $L^p$ summability of the transport density $\sigma_{ib}$, under the assumption that $\,\Omega\,$ satisfies a uniform exterior ball and by supposing that $\,g_2$ is $\lambda$-Lipschitz with $\lambda < 1$ and $C^{1,1}(\bar{\Omega})$. First, we will suppose that $\,\Omega\,$ has a very particular shape, i.e. its boundary is composed of parts of a sphere of radius $r$ (such domains are called $\textit{round polyhedra}$), and then, by an approximation argument, we are able to generalize the result to any domain having a uniform exterior ball.\\

To do that, let us consider the following transport problem

$$\min \left\{\int_{\bar{\Omega}\times\bar{\Omega}} |x-y|\mathrm{d}\gamma+\int_{\partial{\Omega}}g\mathrm{d}(\Pi_{y})_{\#}\gamma
\,:\,\gamma \in \mathcal{M}^+(\bar{\Omega}\times\bar{\Omega})
,\,(\Pi_{x})_{\#}\gamma=f,\,\spt((\Pi_{y})_{\#}\gamma) \subset \partial\Omega \right\}.\quad(\mbox{P2})$$ 
\\
Suppose that $g$ is $\lambda$-Lipschitz with $\lambda < 1$ and set $$ T(x)=\mbox{argmin}\left\{|x-y|+g(y),\,y \in \partial\Omega\right\} \; \mbox{for all} \;x \in \bar{\Omega}.$$
Then, we have the following\,:
\begin{proposition} \label{Prop 2.1}
$T(x)$ is a singleton Lebesgue-almost everywhere.
\end{proposition}
\begin{proof}
Set $$f(x)=\mathrm{min}\{|x-y|+g(y),\,y \in \partial\Omega\}.$$ 
\\
It is clear that $f$ is $1$-Lip, therefore it is differentiable Lebesgue-almost everywhere. Let $x_0$ be in $\,\accentset{\circ}{\Omega}\,$ and suppose that there exist $\,y_0\,$ and $\,y_1 \in \partial\Omega$ such that \\
 $$f(x_0)=|x_0-y_0|+g(y_0)=|x_0-y_1|+g(y_1).$$
As $$f(x) -|x-y_0| \leq  g(y_0)\;\,\mbox{for all}\; x \in \bar{\Omega},$$ \\
 then the function: $ x\mapsto f(x) -|x-y_0| $
reaches a maximum at $x_0$. Hence, if it is 
 differentiable at $x_0 \in \accentset{\circ}{\Omega},$ then 
$\nabla f(x_0)=\frac{x_0-y_0}{|x_0-y_0|}$. In the same way, we get $\nabla f(x_0)=\frac{x_0-y_1}{|x_0-y_1|}$. \\

Hence, we have $\frac{x_0-y_0}{|x_0-y_0|}=\frac{x_0-y_1}{|x_0-y_1|}$, which is a contradiction as $y_1$ is in the half line with vertex $x_0$ and passing through $y_0$, indeed in this case, one has \begin{eqnarray*}
|y_0-y_1| &=& \left|\; \;|x_0-y_0|-|x_0-y_1|\;\;\right| \\
&=& |g(y_0)-g(y_1)| \\
& \leq & \lambda |y_0-y_1|.
\end{eqnarray*}

Note that if $\,\Omega\,$ is convex, we can prove the same result without using the fact that $g$ is $\lambda$-Lip with $\lambda <1$, indeed a half line with vertex in the interior of $\Omega$ cannot intersect $\partial\Omega$ at two different points. $\qedhere$
\end{proof} 
\begin{proposition}\label{Prop 2.2}
If $x \in \Omega$ and $y\in T(x)$, then $(x,y)\cap \partial\Omega=\emptyset.$
  \end{proposition} 
  \begin{proof} Suppose that this is not the case, i.e there exist $x \in \Omega,\,y\in T(x)$ and some point $z \in (x,y) \cap \partial\Omega$. By definition of $T$, we have 
$$ |x-y|+ g(y) \leq |x-z| + g(z). $$
Then $$ |z -y|=|x-y|-|x-z| \leq g(z)-g(y) \leq \lambda |z-y|,$$
which is a contradiction. $\qedhere$
  \end{proof}
  Let $S$ be the set of all the points $x \in \bar{\Omega}\,$ where $\,T(x)$ is a singleton. Then, we have the following\,:
   \begin{proposition} \label{Prop 2.3}
   If $x \in S$ and $y \in [x,T(x)]$, then $y \in S$ and $\,T(y)=T(x)$.\end{proposition}
   \begin{proof} For every $z \neq T(x) \in \partial\Omega$, we have
 \begin{eqnarray*}
  |y-T(x)|+g(T(x))&=&|x-T(x)|-|x-y|+g(T(x)) \\
  & < & |x-z|+g(z)-|x-y| \\
  & \leq &|y-z|+g(z). \qedhere
 \end{eqnarray*} 
\end{proof} 
We observe that if $x \in S$, then the image of $y$ and $x$ through $T$ is the same. This is a well-known principle in optimal transport with distance cost, as $y$ is on the same transport ray as $x$ (i.e. $T(y)=T(x)$).
\begin{proposition} \label{Selector} The multi-valued map $\,T\,$ has a Borel selector function.
 \end{proposition} 
\begin{proof} To prove that $T$ has a Borel selector function, it is enough to show that the graph of $T$ is closed (see for instance Chapter 3 in \cite{13}). Take a sequence $(x_n,y_n)$ in the graph of $T$ such that $(x_n,y_n) \rightarrow (x,y)$. As $\,y_n \in T(x_n)$, then we have
  $$|x_n-y_n|+g(y_n)\leq|x_n-z|+g(z)\;\mbox{ for all }\;z \in \partial\Omega.$$
  Passing to the limit, we get 
  $$|x-y|+g(y)\leq|x-z|+g(z)\;\mbox{ for all} \;z \in \partial\Omega $$ 
 and then, $y \in T(x)$. $\qedhere$
  \end{proof} 
For simplicity of notation, we still denote this selector by $T$.\\ 
\\
  It is not difficult to prove that the plan $\gamma_T:=(Id,T)_{\#}f$
  is the unique minimizer for (P2). In addition, $\gamma_T$ solves the following problem
$$\min\left\{\int_{\bar{\Omega}\times\bar{\Omega}}|x-y|\mathrm{d}\gamma
 \,:\,\gamma \in \Pi(f,(T)_{\#}f) \right\}.$$
 
 For simplicity of notation, we will denote this minimizer by $\gamma$ instead of $\gamma_T$.\\ 
 \\
 Let $\,\sigma$ be the transport density associated with the transport of $f$ into $T_{\#}f$. By the definition of $\sigma$ (see \ref{DefTrans}), we have that for all $ \phi \in C(\bar{\Omega})$ 
\begin{eqnarray*}
<\sigma,\phi> &=& \int_{\bar{\Omega}\times\bar{\Omega}}\int_0^1|x-y|
  \phi((1-t)x + ty)\mathrm{d}t \mathrm{d}\gamma(x,y) \\
 &=&\int_{\bar{\Omega}}\int_0^1|x-T(x)|\phi((1-t)x+tT(x))f(x)\mathrm{d}t \mathrm{d}x.
\end{eqnarray*}
  Then $$\sigma=\int_0^1 \mu_t \mathrm{d}t,$$ where $$<\mu_t,\phi>:=\int_{\bar{\Omega}}|x-T(x)|\phi(T_t(x))f(x)\mathrm{d}x\;\;\mbox{for all}\,\;\varphi \in C(\bar{\Omega})$$  and
  $$T_t(x):=(1-t)x+tT(x) \mbox{ for all} \;x \in \bar{\Omega}.$$ \\

 Notice that in the definition of $\mu_t$, differently from what done in \cite{9}, we need to keep the factor $|x-T(x)|$, which will be essential in the estimates. In addition, we have that $\,\mu_t \ll \mathcal{L}^d$ as soon as one has $\,f\ll \mathcal{L}^d$ (see \cite{9}).\\
  
  Suppose that the boundary of $\,\Omega\,$ is a union of parts of sphere of radius $r>0$. Then, we have the two following propositions, whose proofs, for simplicity of exposition, are postponed to 
  Section \ref{5}. 
  \begin{proposition} \label{Prop. 2.7} Suppose that $g$ is in $\,C^2(\bar{\Omega})$. Then, there exists a negligible closed set $N$ in 
  $\bar{\Omega}$ such that $\,\accentset{\circ}{\Omega} \backslash N \subset S.$ Moreover, $T$ is a $\,C^1$ function on $\,\accentset{\circ}{\Omega} \backslash N$.\\
  \end{proposition}

Now, we want to give an explicit formula of $\mu_t$ in terms of $f$ and $T$. Let $\phi$ be in $C(\bar{\Omega})$, then we have
  $$\int_{\bar{\Omega}}\phi(y)\mathrm{d}\mu_t(y)=\int_{U_t}\phi(T_t(x))\left|x-T(x)\right| f(x) \mathrm{d}x,$$ where $\,U_t:=\{x \in \accentset{\circ}{\Omega} \backslash N:\,T_t(x) \in \accentset{\circ}{\Omega} \backslash N\}.\quad$\\
  \\
Take a change of variable $y=T_t(x)$. By Propositions \ref{Prop 2.3} \& \ref{Prop. 2.7}, we infer that \\
 $$x=\frac{y-tT(y)}{1-t}\,\;\,\mbox{and}\;\, \left|x-T(x)\right|=\frac{|y-T(y)|}{1-t}.$$
 
  Consequently, we have \\
 $$ \int_{\bar{\Omega}}\phi(y)\mu_t(y)\mathrm{d}y=\int_{\bar{\Omega}} \phi(y) \frac{|y-T(y)|}{1-t}f\left(\frac{y-tT(y)}{1-t}\right)|J(y)|1_{V_t}(y)\mathrm{d}y, $$\\
   where $V_t:=T_t(U_t)\;$ and
  $\;J(y):=\frac{1}{\mbox{det}(DT_t(x))}$.\\
  
  Finally, we get $$ \mu_t(y)= \frac{|y-T(y)|}{1-t}f\left(\frac{y-tT(y)}{1-t}\right)|J(y)|1_{V_t}(y)\;\,\mbox{for a.e}\;\,y \in \Omega.$$
  \\
Notice that for all $ y \in V_t$, we have $|y-T(y)|\leq (1-t)  l(y)$
where $l(y)$ is the length of the transport ray containing $y$, i.e
$$
 l(y):=\mathrm{sup}\left\{|x-T(x)|:\,T(x)=T(y),\,x \in \bar{\Omega}\cap\{s y + (1-s) T(y),\, s\geq 1\}\right\}.$$ 
   \begin{proposition} \label{Imp} There exist a constant $L:=L(d,\mbox{diam}(\Omega),\lambda,r,D^2 g)>0$ and a compact set $ K:=\{y \in \bar{\Omega},\,d(y ,\partial\Omega)\geq L\}$ such that for a.e $ x \in \Omega $, if $\;T_t(x) \in \Omega \backslash K$ then we have the following estimate
  $$|\det(DT_t(x))| \geq C (1-t),$$
  where $\,C:=C(d,\mbox{diam}(\Omega),\lambda,r,D^2 g)>0.$
  \end{proposition}
  We are now ready to prove the $L^p$ summability of the transport density $\sigma$. Then, under the assumption that $\Omega$ is a round polyhedron, we have the following result.
  \begin{proposition} \label{infinityestimates} 
   Suppose that $\,\Omega\,$ is a round polyhedron.\,\,Then, the transport density $\sigma$ belongs to $L^{\infty}(\Omega)$ provided that $f \in L^{\infty}(\Omega)$.
  \end{proposition}
  \begin{proof} 
  By Proposition \ref{Imp}, we have
  \begin{eqnarray*}
  \parallel \sigma \parallel _{L^{\infty}(\Omega  \backslash K)}  &:=& \sup_{y \,\in\, \Omega  \backslash K}\left(\int_0^1 \mu_t(y)\mathrm{d}t\right)\\
 &=&\sup_{y \,\in \,\Omega  \backslash K }\left(\int_0^{1-\frac{|y-T(y)|}{l(y)}}\frac{|y-T(y)|f(\frac{y-tT(y)}{1-t})}{(1-t)|\mbox{det}(DT_t(x))|}1_{V_t}(y)\mathrm{d}t\right) \\
 & \leq & C^{-1}  \parallel f \parallel_{L^{\infty}(\Omega)}\left(\int_0^{1-\frac{|y-T(y)|}{l(y)}}\frac{|y-T(y)|}{(1-t)^{2}}\mathrm{d}t\right).
\end{eqnarray*}
  Yet, it is easy to see that
  $$\int_0^{1-\frac{|y-T(y)|}{l(y)}}\frac{|y-T(y)|}{(1-t)^{2}}\mathrm{d}t\, \leq \, \mbox{diam}(\Omega).$$ 
    Then, $$\parallel \sigma \parallel_{L^{\infty}(\Omega  \backslash K )}\, \leq C^{-1} \mbox{diam}(\Omega)\parallel f \parallel_{L^{\infty}(\Omega)}.$$ \\
    
   On the other hand, by Minkowski's inequality
  $$\parallel \sigma \parallel_{L^{\infty}(K)}\,\leq \int_0^1 \parallel \mu_t\parallel_{L^{\infty}(K)}\mathrm{d}t. $$ \\
  Notice that for a.e $y \in \Omega$, if $\mu_t(y)\neq 0 $ then $|y-T(y)|\leq (1-t) l(y).$ But for any $y \in K$, we have $$\frac{|y-T(y)|}{l(y)}\geq \frac{d(y,\partial\Omega)}{\mathrm{diam}(\Omega)} \geq \frac{L}{\mathrm{diam}(\Omega)}:=L_1.$$ 
  
  Then, $$\parallel \sigma \parallel_{L^{\infty}(K)}\,\leq \int_0^{1-L_1}\parallel \mu_t \parallel_{L^{\infty}(K)}\mathrm{d}t.$$ By \cite{9}, we have $$ \parallel \mu_t\parallel_{L^{\infty}(K)}\,\leq (1-t)^{-d} \mbox{diam}(\Omega)  \parallel f \parallel_{L^{\infty}(\Omega)}.$$
  
Finally, we get $$ \parallel \sigma \parallel_{L^{\infty}(\Omega)} \,\leq  \max\left\{\int_0^{1-L_1}(1-t)^{-d}\mathrm{d}t,C^{-1}\right\}\mbox{diam}(\Omega)\parallel f \parallel_{L^{\infty}(\Omega)},$$
where $L_1$ and $C$ are two strictly positive constants depending only on $d,\,\mbox{diam}(\Omega),\,\lambda,\,r$ and $D^2g. \qedhere$
\end{proof}
 \begin{proposition} \label{Ourtheorem} 
Let $\,\Omega\,$ be a round polyhedron and suppose $f \in L^p(\Omega)$ for some $p \in [1,+\infty]$. Then, the transport density $\,\sigma$ also belongs to $L^p(\Omega)$.
 \end{proposition}
 \begin{proof}
 We observe that as the transport is between $f$ and $T_{\#}f$, then the transport density $\sigma$ linearly depends on $f\,$: in this case, $L^p$ estimates could be obtained via interpolation as soon as one has $L^1$ and $L^\infty$ estimates (see for instance \cite{Marc}). $\qedhere$
\end{proof}
\begin{remark}
The same proof as in Proposition \ref{infinityestimates} could also be adapted to proving Proposition \ref{Ourtheorem}, but a suitable use of a H\"older inequality would be required.\\
\end{remark}
We will now generalize, via a limit procedure, the result of Proposition \ref{Ourtheorem} to arbitrary domain having a uniform exterior ball. But before that, we will give a definition of such a domain. 
\begin{definition} We say that $\,\Omega\,$ has a uniform exterior ball of radius $r>0$ if for all $\,y \in \partial\Omega$, there exists some $x \in \mathbb{R}^d$ such that $B(x,r)\cap \Omega=\emptyset\,$ and $\,|x-y|=r$.
\end{definition}

We suppose that $\Omega$ is such a domain, then we have the following\,: 
\begin{proposition} \label{Prop. 6.1}  
The transport density $\,\sigma$ between $\,f$ and $\,T_{\#} f$ belongs to $L^p(\Omega)$ provided that $f \in L^p(\Omega)$ and $\,g$ is $\lambda$-Lip with $\lambda < 1$ and $\,C^{1,1}(\bar{\Omega})$.
\end{proposition}
\begin{proof}
This proposition can be proven using the same ideas as in Proposition 3.4 in \cite{DweSan}. To do that, take a sequence of domains $(\Omega_k)_k$ such that\,: the boundary of each $\Omega_k$ is a union of parts of sphere of radius $r$, $\partial\Omega_k\rightarrow \partial\Omega\;$ in the Hausdorff sense and $\,\Omega \subset \Omega_k \subset \tilde{\Omega}\,$ for some large compact set $ \tilde{\Omega}$.\\

First of all, we extend $f$ by $0$ outside $\Omega$ and we suppose that $g \in C^2(\bar{\Omega})$. Let $\gamma_k$ be an optimal transport plan between $f$ and $(T^k)_{\#}f$, i.e $\gamma_k$ solves $$\min \left\{\int_{\tilde{\Omega} \times \tilde{\Omega}}|x-y| \mathrm{d} \gamma\,:\,\gamma \in \Pi(f,(T^k)_{\#}f)\right\},$$ 
 where $\,T^k(x)=\mbox{argmin}\{|x-y| + g(y),\,y \in \partial\Omega_k\}.$\\

Let $\sigma_k$ be the transport density associated with the optimal transport plan $\gamma_k$. From Proposition \ref{Ourtheorem}, we have
$$\sigma_k \in L^p(\Omega_k)$$
 and $$\parallel \sigma_k\parallel _{L^p(\Omega_k)}\,\leq C \parallel f \parallel_{L^p(\Omega)},$$ \\
 where $\,C:=C(d,\mbox{diam}(\Omega),\lambda,r,D^2 g)$.\\
 
 Then, up to a subsequence, we can assume that $\sigma_k \rightharpoonup \sigma$ weakly in $L^p(\tilde{\Omega})$. Moreover, we have the following estimate 
$$\parallel \sigma \parallel_{L^p(\Omega)}\,\leq \liminf\limits_k  \parallel \sigma_k \parallel _{L^p (\Omega_k)}\,\leq C \parallel f \parallel _{L^p(\Omega)}.$$ 
Hence, it is sufficient to show that this $\sigma$ is in fact the transport density associated with the transport of $f$ into $(T)_{\#}f.$ \\

Firstly, we observe that $(T_k(x))_k$ converges, up to a subsequence, to a point $y \in \partial \Omega$ such that $y \in \mbox{argmin}\{|x-z| + g(z),\,z \in \partial\Omega\}$. Since this point is unique for a.e. $x$, we get (with no need of passing to a subsequence)\,: 
 $$T^k(x) \rightarrow T(x)$$
 and 
  $$(T^k)_{\#}f \rightharpoonup (T)_{\#}f\;\, \mbox{in the sense of measures}.$$
\\
By Theorem 5.20 in \cite{11}, we get that 
 $$\gamma_k \rightharpoonup \gamma, $$ where $\gamma $ solves $$ \min \left\{ \int_{\bar{\Omega}\times\bar{\Omega}}|x-y|\mathrm{d}\gamma\,:\,\gamma\in \Pi(f,(T)_{\#}f)\right\}.$$ \\
Let $\sigma_{\gamma}$ be the unique transport density between $f$ and $(T)_{\#}f$. As $\gamma_k  \rightharpoonup \gamma$, we find that $\sigma_k \rightharpoonup \sigma_{\gamma}$ (see \ref{DefTrans}). Consequently, $\sigma_{\gamma}=\sigma \in L^p(\Omega)$ and we have the following estimate $$\parallel \sigma_{\gamma}\parallel _{L^p(\Omega)} \,\leq C  \parallel f \parallel _{L^p(\Omega)},$$
where $\,C:=C(d,\mbox{diam}(\Omega),\lambda,r,D^2g). $\\

The approximation of a $C^{1,1}$ function $g$ with smoother functions is also standard. Then, it is not difficult to check again that our result is still true for  a $C^{1,1}$ function $g$. $\qedhere$
\end{proof}
\section{The case $g=0$}
\label{4} 
In the particular case $g=0$, we are able to prove Proposition \ref{Imp} via a geometric argument which will not be available for the general case.

\begin{lemma} \label{Lem 1} 
Let $ P_{\partial\Omega}$ be the projection on the boundary of $\,\Omega$, i.e 
$$ P_{\partial\Omega}(x):=\mbox{argmin}\left\{|x-y|,\,y \in \partial\Omega\right\}\;\mbox{for all}\;\;x.$$
Then, $P_{\partial\Omega}$ is the gradient of a convex function. In addition, if $\,\Omega\,$ has a uniform exterior ball of radius $r>0$, then for a.e $x \in \Omega$,
 the positive symmetric matrix $DP_{\partial\Omega}(x)$ has $d-1$ eigenvalues larger than $r/r+d(x,\partial\Omega)$.
 \end{lemma} 
\begin{proof}  
Set $$u(x):=\mathrm{sup}\left\{x \cdot y-\frac{1}{2}|y|^2,\,y \in \partial\Omega\right\}.$$
As we can rewrite $u(x)$ as follows
$$u(x)=\sup\left\{-\frac{1}{2}|x-y|^2+\frac{1}{2}|x|^2,\,y\in\partial\Omega\right\},$$
then the supremum is  attained at $P_{\partial\Omega}(x)$ and
$\nabla u(x)=P_{\partial\Omega}(x)$ for a.e $x$.\\

On the other hand, take $x_0 \in \accentset\circ{\Omega}$ and let $y_0$ be the center of a ball $B(y_0,r)$ such that $B(y_0,r)\cap \Omega=\emptyset\,$ and $\,|y_0 - P_{\partial\Omega}(x_0)|=r$. Then $\,x_0,\,P_{\partial\Omega}(x_0) \mbox{ and}\;\,y_0 \mbox{ are aligned}$. Indeed, if not, we get $|x_0-y_0| < |x_0 - P_{\partial\Omega}(x_0)|+r$, but
 $ |x_0-y_0|=|x_0-z|+|z-y_0|$ for some $z \in [x_0,y_0] \cap \partial\Omega$, which is a contradiction as $|x_0 - z|\geq |x_0 - P_{\partial\Omega}(x_0)|\;\mbox{and}\;|z-y_0| \geq r.$ \\
 
 Moreover, we have
\begin{eqnarray*}
u(x)&=&\sup\left\{\frac{1}{2}|x|^2-\frac{1}{2}|x-y|^2,\,y\in \partial\Omega\right\}\\
&\geq &\frac{1}{2}|x|^2-\frac{1}{2}|x-y^{\star}|^2,\;\mbox{for some}\;\,y^{\star}\in [x,y_0]\cap \partial\Omega \\
 &\geq & \frac{1}{2}|x|^2-\frac{1}{2}(|x-y_0|-r)^2 
\,\,:=v(x).\\
\end{eqnarray*}
As $u(x_0)= v(x_0)$, then the function: $x\mapsto u(x)-v(x)$ has a minimum
 at $x_0$. Hence, we get that $\,D^2u(x_0) \geq D^2 v(x_0)$
and the eigenvalues of $D^2u(x_0)$ are bounded from below by those of $D^2v(x_0)$. Yet, it is easy to show that $$D^2 v(x_0)=\frac{r}{r+d(x_0,\partial\Omega)}\left(I-e(x_0)\otimes e(x_0)\right),$$
where $\,e(x_0):=(x_0-y_0)/|x_0-y_0|$. Then, we conclude by observing that
the eigenvalues of this matrix are $\,0\,$ and $\,r/r+d(x_0,\partial\Omega)$ (with multiplicity $d-1$).$\qedhere$ \end{proof}
Set $$P_t(x):=(1-t)x + tP_{\partial\Omega}(x).$$ 
By Lemma \ref{Lem 1}, we have 
              $$\mbox{det}(DP_t(x)) \geq (1-t)\left(1-t+t\frac{r}{r+d(x,\partial\Omega)}\right)^{d-1}. $$ \\
              Set $\,y:=P_t(x)$, then the Jacobian at $y$ satisfies the following estimate 
\begin{eqnarray*}
J(y):=\frac{1}{\mbox{det}(DP_t(x))} &\leq & \frac{1}{1-t}\left(\frac{r+d(x,\partial\Omega)}{r + (1-t)d(x,\partial\Omega)}\right)^{d-1}\\
            & = & \frac{1}{(1-t)^d}\left(\frac{(1-t)r+d(y,\partial\Omega)}{r + d(y,\partial\Omega)}\right)^{d-1}.\\
\end{eqnarray*}
Suppose $f \in L^{\infty}(\Omega)$ and let $\sigma$ be the transport density between $f$ and $(P_{\partial\Omega})_{\#}f$, then we have the following pointwise inequality\\
$$ \sigma(y) \leq \, \parallel f \parallel_{L^{\infty}(\Omega)} \int_0^{1-\frac{d(y,\partial\Omega)}{l(y)}}\frac{d(y,\partial\Omega)}{(1-t)^{d+1}}\left(\frac{(1-t)r+d(y,\partial\Omega)}{r + d(y,\partial\Omega)}\right)^{d-1}\mathrm{d}t,$$ \\
where $\,l(y)$ is the length of the transport ray containing $\,y$.\\ 

Hence, \begin{eqnarray*}
\sigma(y) & \leq & \, C(d) \parallel f \parallel_{L^{\infty}(\Omega)} \int_0^{1-\frac{d(y,\partial\Omega)}{l(y)}}\frac{d(y,\partial\Omega)}{(1-t)^{d+1}} \frac{(1-t)^{d-1}r^{d-1}+(d(y,\partial\Omega))^{d-1}}{(r + d(y,\partial\Omega))^{d-1}}\mathrm{d}t \\ \\
& = & \frac{C(d) \parallel f \parallel_{L^{\infty}(\Omega)}}{(r + d(y,\partial\Omega))^{d-1}} \left( r^{d-1} d(y,\partial\Omega) \int_0^{1-\frac{d(y,\partial\Omega)}{l(y)}}\frac{1}{(1-t)^{2}}\mathrm{d}t + (d(y,\partial\Omega))^d \int_0^{1-\frac{d(y,\partial\Omega)}{l(y)}}\frac{1}{(1-t)^{d+1}}\mathrm{d}t \right). \\
\end{eqnarray*}
But, it is easy to see that \begin{multline*}
r^{d-1} d(y,\partial\Omega) \int_0^{1-\frac{d(y,\partial\Omega)}{l(y)}}\frac{1}{(1-t)^{2}}\mathrm{d}t \,\,+ \,\,(d(y,\partial\Omega))^d \int_0^{1-\frac{d(y,\partial\Omega)}{l(y)}}\frac{1}{(1-t)^{d+1}}\mathrm{d}t \\ \\
\leq \; C:=C(d,r,\mbox{diam}(\Omega)). 
\end{multline*}

Consequently, $$ \sigma(y) \leq \frac{C \parallel f \parallel_{L^{\infty}(\Omega)}}{(r + d(y,\partial\Omega))^{d-1}}.$$ \\

This provides a very useful and pointwise estimate on $\sigma$. It shows that $\sigma$ is bounded as soon as $r>0$, or we are far from the boundary $\partial\Omega$. As a particular case, we get the results of Section \ref{2} in the case $g=0$ whenever $r>0$.\\

By interpolation, we get also that $\sigma$ belongs to $L^p(\Omega)$ provided that $f \in L^p(\Omega)\,$ (see \cite{Marc}). 
                \section{Technical proofs} \label{5}
                In this section, we want to give the proofs of Propositions \ref{Prop. 2.7} \& \ref{Imp}. First of all, suppose that $\,\Omega\,$ is a round polyhedron and set 
                 $$\Omega_i:=\{x=(x_1,x_2,.....,x_d)\in \bar{\Omega}:\,T(x) \in F_i\},$$        
  where $\,T\,$ is the Borel selector function, which was mentioned earlier in Proposition \ref{Selector}, $F_i\subset \partial B(b_i,r)$ is the $i$th part in the boundary of $\Omega$, contained in a sphere centered at $b_i$ and with radius $r>0$. \\
  
  Then, we have the following\,:
 \begin{proposition} \label{Prop. 5.1} For all $x \in \accentset{\circ}{\Omega}$, there does not exist $i\neq j$ such that $\,T(x) \in F_i\cap F _j.$
  \end{proposition}
\begin{proof} Suppose that this is not the case at some point  $x \in \accentset{\circ}{\Omega}$, i.e there exist two different faces $F_i$ and $F_j$ such that $ T(x)\in F_i \cap F_j$. By Proposition \ref{Prop 2.2}, the segment $[x,T(x)]$ cannot intersect the boundary of $\Omega$ at another point $z \neq T(x)$. Then taking into account the geometric form of $\Omega$ (see the proof of the Proposition \ref{Prop. 6.1}), we can assume that there exist $\dot{\gamma_1}(0) $ and $\dot{\gamma_2}(0)$ two tangent vectors in $T(x)$ on $F_i$ and $F_j$ respectively in such a way that the angle between them is less than $180^{\circ}$ ($\gamma_1$ and $\gamma_2$ 
  are two curves plotted on $F_i$ and $F_j$ respectively) and $$ x-T(x)= \alpha \dot{\gamma_1}(0) + \beta \dot{\gamma_2}(0)$$
  for some two positive constants $\alpha$ and $\beta$.
  \\

  Set $$f_1(t)=|x-\gamma_1(t)| + g(\gamma_1(t))$$
and $$f_2(t)=|x-\gamma_2(t)| + g(\gamma_2(t)).$$ \\
By optimality of $\,T(x)=\gamma_1(0)=\gamma_2(0)$, we can deduce that
$\,\dot{f_1}(0)\geq 0 \;\,\mbox{and}\,\;\dot{f_2}(0)\geq 0$. Hence, we have 
$$ -\frac{x-T(x)}{|x-T(x)|}\cdot\dot{\gamma_1}(0)+\nabla g(T(x))\cdot\dot{\gamma_1}(0)\geq 0$$
and $$ -\frac{x-T(x)}{|x-T(x)|}\cdot\dot{\gamma_2}(0)+\nabla g(T(x))\cdot\dot{\gamma_2}(0)\geq 0 .$$ \\
We multiply the first inequality by $\alpha$, the second one by $\beta$ and we take the sum, we get\\
$$-|x-T(x)|+\nabla g(T(x))\cdot(x-T(x))\geq 0 $$
and $$1\leq |\nabla g(T(x))|\leq \lambda, $$ which is a contradiciton.
$ \qedhere$ \end{proof} 
\begin{proposition} \label{Prop. 5.2} For every $\,x \in \Omega_i \cap S \cap \accentset\circ{\Omega}$, there exists a neighborhood of $x$ contained in $\Omega_i $.
\end{proposition}
\begin{proof} Suppose that this is not the case at some point $x$. Then, there exists a sequence $(x_n)_n$ such that $x_n\rightarrow x$ and $ T(x_n)\in F_j$ for some $j \neq i$. 
Yet, up to a subsequence, we can assume that $T(x_n)\rightarrow y \in F_j$. By definition of $T$, we have
$$|x_n-T(x_n)|+g(T(x_n))\leq |x_n-z|+g(z) \;\mbox{for all}\;z \in \partial\Omega.$$
Passing to the limit, we get
$$|x-y|+g(y)\leq |x-z|+g(z)\;\mbox{for all}\; z \in \partial\Omega, $$
which is in contradiction with Proposition \ref{Prop. 5.1}. $\qedhere$ \end{proof}
Consider $\Omega_1$ (eventually it will be the same for the other $\Omega_i$)
and recall that $$\Omega_1:=\{x=(x_1,....,x_d)\in \bar{\Omega}:\,T(x) \in F_1\}.$$
 Suppose that Proposition \ref{Prop. 2.7} is true and fix $x \in \Omega_1 \cap \accentset{\circ}{\Omega} \backslash N $. After a translation and rotation of axis, we can suppose that the tangent space at $T(x)$ on $F_1$ is contained in the plane $x_d=0$ and that there exists $\varphi:U \mapsto \mathbb{R}$, where $U \subset \mathbb{R}^{d-1}$, a parameterization of $F_1$, i.e for any $z:=(z_1,...,z_d) \in F_1,$ we have $\bar{z}:=(z_1,...,z_{d-1})\in U$ and $z_d=\varphi(\bar{z})$ (notice that an explicit formula of $\varphi$ is not needed for the sequel).\\
 
 For simplicity of notation, we denote $\alpha(x):=|x-T(x)|.$\\
 \\
 Set  
$$f(z)=\sqrt{|\bar{x}-z|^2 + (x_d-\varphi(z))^2}+ g(z,\varphi(z))\;\mbox{for all} \;z \in U.$$ \\
For any $i \in \{1,....,d-1\}$, \\
$$\frac{\partial f}{\partial z_i}(z)=\frac{(z_i-x_i)+(\varphi(z)-x_d)\frac{\partial \varphi}{\partial z_i}(z)}{\sqrt{|\bar{x}-z|^2 + (x_d-\varphi(z))^2}}+
\frac{\partial g}{\partial z_i}(z,\varphi(z))+\frac{\partial g}{\partial z_d}(z,\varphi(z))\frac{\partial \varphi}{\partial z_i}(z).$$ \\
\\
 Set $\,T(x)=(\bar{T}(x),\varphi(\bar{T}(x)))$, 
 where $\,\bar{T}(x):=(T_1(x),...,T_{d-1}(x))$. Then, we have\\ $$ \bar{T}(x)=\mathrm{argmin}\{f(z),\;z \in U\}.$$\\
By Proposition \ref{Prop. 5.1}, $\bar{T}(x) \in \accentset{\circ}{U}$. Hence, we get \\ $$\frac{\partial f}{\partial z_i}(\bar{T}(x))=0\;\;\mbox{for all}\;\,i \in \{1,....,d-1\}
$$ 
or equivalently, \\
\begin{equation} \label{E1}
\frac{T_i(x)-x_i}{\alpha(x)}+ \frac{\partial g}{\partial z_i}(T(x)) +
\frac{\varphi(\bar{T}(x)) - x_d}{\alpha(x)}\frac{\partial \varphi}{\partial z_i}(\bar{T}(x)) +
\frac{\partial g}{\partial z_d}(T(x))\frac{\partial \varphi}{\partial z_i}(\bar{T}(x))=0 \\
\end{equation}
for all $\;i \in \{1,....,d-1\}$.
\\ \\
 By Propositions \ref{Prop. 2.7} \& \ref{Prop. 5.2}, the equality in (\ref{E1}) holds in a neighborhood of $x$. Then, differentiating (\ref{E1}) with respect to $x_j$ and taking into account the fact that in this new system of coordinates we have 
  $$\frac{\partial \varphi}{\partial z_i}(\bar{T}(x))=0\;\mbox{ for all}\;\, i \in \{1,....,d-1\},$$ we get
  \begin{align}\label{E2}
  \nonumber
  \frac{\partial T_i}{\partial x_j}(x)+ \frac{(T_i(x)-x_i)}{(\alpha(x))^2}\sum_{k=1}^{d-1}(x_k - T_k(x))\frac{\partial T_k}{\partial x_j}
+ \alpha(x) \sum_{k=1}^{d-1}\frac{\partial^2g}{\partial z_i\partial z_k}(T(x))\frac{\partial T_k}{\partial x_j}(x)
\\ 
\,\;+ \;\;\,(\varphi(\bar{T}(x)) - x_d)\sum_{k=1}^{d-1}\frac{\partial^2\varphi}{\partial z_i\partial z_k}(\bar{T}(x))\frac{\partial T_k}{\partial x_j}(x) 
+ \alpha(x) \frac{\partial g}{\partial z_d}(T(x))\sum_{k=1}^{d-1}\frac{\partial^2 \varphi}{\partial z_i\partial z_k}(\bar{T}(x))\frac{\partial T_k}{\partial x_j}(x)
  \end{align}
  $$
=\delta_{ij} + \frac{(T_i(x)-x_i)(x_j - T_j(x))}{(\alpha(x))^2}$$
  $\mbox{for all}\;i,\,j \in \{1,...,d-1\}. $ \\

  On the other hand, we have \\
\\
$DT_t(x)= (1-t)I + tDT(x)=\begin{pmatrix}
 1-t + t \frac{\partial T_1}{\partial x_1} & t\frac{\partial T_1}{\partial x_2} & ...& t\frac{\partial T_1}{\partial x_d} \\
 \\
 t\frac{\partial T_2}{\partial x_1} & 1-t + t \frac{\partial T_2}{\partial x_2} &...& t \frac{\partial T_2}{\partial x_d} \\
 \\
 \\
 ...
 \\
 \\
 0 & ...&0& 1-t
 \\
 \end{pmatrix}.$
\\ 
\\ 
Then,\\ $$|DT_t(x)|=(1-t)|\mbox{det}(A)|,$$ \\
 where $\,A=\left((1-t)\delta_{ij} + t\frac{\partial T_i}{\partial x_j}(x)\right)_{i,j=1,....,d-1}.$ \\ 
 \\
\\
Set $$F=\left(\delta_{ij}  - \frac{(x_i-T_i(x))(x_j-T_j(x))}{(\alpha(x))^2}\right)_{ij}$$
and \\$$ N=\left(\alpha(x)\frac{\partial^2g}{\partial z_i\partial z_j}(T(x))+(\varphi(\bar{T}(x))-x_d)\frac{\partial^2\varphi}{\partial z_i\partial z_j}(\bar{T}(x))\\
 \\+\alpha(x)\frac{\partial g}{\partial z_d}(T(x))\frac{\partial^2\varphi}{\partial z_i\partial z_j}(\bar{T}(x))\right)_{ij}.$$\\ \\
Suppose that $F+N$ is invertible for a.e $\,x\in\Omega$ (see Proposition \ref{Inv} below). From (\ref{E2}), we observe that $$\left(\frac{\partial T_i}{\partial x_j}(x)\right)_{i,j=1,....,d-1}=(F+N)^{-1} F.$$ 
Hence,
$$A = (1-t)I+t(F+N)^{-1}F =(F+N)^{-1}(F+(1-t)N)$$
and $$\frac{1}{|\mbox{det}(A)|}=\frac{|\mbox{det}(F+N)|}{|\mbox{det}(F+(1-t)N)|}.$$ \\
As $\,D^2 \varphi(\bar{T}(x)) =\frac{-1}{r}I$, then
 $$\pm(F+N) \leq C(d,\mbox{diam}(\Omega),\lambda,r,D^2 g) I$$ and $$|\mbox{det}(F+N)| \leq C(d,\mbox{diam}(\Omega),\lambda,r,D^2 g).$$ \\
By Proposition \ref{Prop 2.3}, we have \\
 $$(1-t)N=\left(\alpha(y)\frac{\partial^2g}{\partial z_i\partial z_j}(T(y))+(\varphi(\bar{T}(y))-y_d)\frac{\partial^2\varphi}{\partial z_i\partial z_j}(\bar{T}(y))+\alpha(y)\frac{\partial g}{\partial z_d}(T(y))\frac{\partial^2\varphi}{\partial z_i\partial z_j}(\bar{T}(y))\right),$$\\
 where $\,y:=T_t(x)$. \\ \\
 From (\ref{E1}), we have
  $$\frac{T_i(x) - x_i}{|x-T(x)|} + \frac{\partial g}{\partial z_i}(T(x))=0\;\;\mbox{for any}\;\;i \in \{1,...,d-1\}. $$
 \\
 Then, $$|\bar{x} - \bar{T}(x)| \leq \lambda |x-T(x)|. $$
But,
\begin{eqnarray*}
<F\xi,\xi> &=& |\xi|^2 - \left( \frac{\bar{x} -\bar{T}(x)}{|x-T(x)|}\cdot\xi\right)^2 \\
&\geq & |\xi|^2 - \frac{|\bar{x} -\bar{T}(x)|^2}{|x-T(x)|^2}|\xi|^2\\
  &\geq & (1-\lambda^2)|\xi|^2.
 \end{eqnarray*} \\
 Hence, there exists some $C(\lambda) > 0$ such that $F \geq C(\lambda) I$. In addition, it is easy to observe that $$\frac{1}{2}F + (1-t)N
 \geq \left(\frac{C(\lambda)}{2} + C_1(d,\lambda,r,D^2g) \alpha(y)\right) I$$ \\
 for some $\,C_1(d,\lambda,r,D^2g) < 0 $.\\
\\
 Now, suppose that $$|y-T(y)| \leq \frac{C(\lambda)}{-2 C_1(d,\lambda,r,D^2g)}.$$
 Then $$ F +(1-t)N \geq \frac{1}{2}F \geq \frac{C(\lambda)}{2}I$$
 and $$\mbox{det}(F+(1-t)N) \geq C(d,\lambda)>0.$$ \\
 \begin{lemma} We have $\,|y-T(y)|\leq \frac{1+\lambda}{1-\lambda} d(y,\partial\Omega)\;\,\mbox{for all }\; y \in \Omega.$ \end{lemma}
 \begin{proof}
 By optimality of $\,T(y)$, we have $$|y-T(y)|+g(T(y)) \leq |y-P_{\partial\Omega}(y)|+g(P_{\partial\Omega}(y)).$$ 
As $g$ is $\lambda$-Lip, then $$|y-T(y)|\leq d(y,\partial\Omega)+\lambda|T(y)-P_{\partial\Omega}(y)|.$$
Yet, $$|T(y)-P_{\partial\Omega}(y)| \leq |y-T(y)| +d(y,\partial\Omega). \qedhere $$
\\
 \end{proof}
Set $$L:=\frac{(1-\lambda)C(\lambda)}{-2(1+\lambda)C_1(d,\lambda,r,D^2g)} >0$$
and $$ K:=\{y \in \bar{\Omega}:\, d(y ,\partial\Omega)\geq L\}.$$\\
Hence, there exists a compact set $K$ such that for a.e $x \in \Omega$, if $\;y:=T_t(x) \in \Omega\backslash K$ then we have the following estimate
  $$|\mbox{det}(DT_t(x))| \geq C(1-t),$$ 
  where $\,C:=C(d,\mbox{diam}(\Omega),\lambda,r,D^2 g)>0.$ \\ 
  \\
  
 Now, we introduce the proof of the Proposition \ref{Prop. 2.7}.\\
\begin{proof}
Fix $a  \in \accentset{\circ}{\Omega} \cap S \cap \Omega_1$ and, without loss of generality, suppose that the tangent space at $T(a)$ on $F_1$ is contained in the plane $x_d=0$. Let $\varphi$ be a parametrization of $F_1$ and for any $\,i \in \{1,...,d-1\}$, set $$h_i(x,y):=\frac{y_i-x_i}{\sqrt{|\bar{x}-y|^2  + (x_d-\varphi(y))^2}}\;+\; \frac{\partial g}{\partial z_i}(y,\varphi(y)) \;+\; \frac{\varphi(y) - x_d}{\sqrt{|\bar{x}-y|^2 + (x_d-\varphi(y))^2} } \frac{\partial \varphi }{\partial z_i}(y)$$
 $$+\; \frac{\partial g}{\partial z_d}(y, \varphi(y)) \frac{\partial \varphi }{\partial z_i}(y)$$ for all $(x,y)\in \accentset\circ{\Omega} \times \accentset{\circ}{U}$.
\\
\\
Set $h:=(h_i)_{i}$, then it is easy to see that $h \in C^1(\accentset{\circ}{\Omega} \times \accentset{\circ}{U},\mathbb{R}^{d-1})$. By Proposition \ref{Inv}, we can assume that the matrix $ (\frac{\partial h_i}{\partial y_j}(a,\bar{T}(a)))_{1\leq i,j \leq d-1}$ is invertible. As $\,h(a,\bar{T}(a))=0$, then by the Implicit Function Theorem, there exist an open neighborhood $K$ of $(a,\bar{T}(a))$ in $\accentset{\circ}{\Omega} \times \accentset{\circ}{U}$, a neighborhood $V$ of $a$ in $\accentset{\circ}{\Omega}$ and a function $q: V \rightarrow \mathbb{R}^{d-1}$ of class $C^1$ such that for all $(x,y) \in K$, we have $$h(x,y)=0 \Leftrightarrow y=q(x).$$ \\
By Proposition \ref{Prop. 5.2} and the fact that $\,T\,$ is continuous at $\,a$, we infer that there exists a small open neighborhood $v(a)\subset \Omega_1$ of $a\,$ such that $(x,\bar{T}(x))\in K$ for every $x \in v(a)$. But $h(x,\bar{T}(x))=0$ for every $x\in v(a)$, then $\bar{T}(x)=q(x)$ and $T$ is a $C^1$ function on $v(a)$. Moreover, we can assume also that $v(a) \subset S$. 
Indeed, if this is not the case, then there exists a sequence $(a_n)_n$ such that $a_n \rightarrow a$ and for all $n$, $a_n \notin S$ (i.e for all $n$, there exist $z_n,\,w_{n} \in \mbox{argmin}\left\{|a_n-y|+g(y),\,y \in \partial\Omega\right\}$ such that $z_n \neq w_n$).
As $\,a \in S$, then $(z_{n})_n$ and $(w_n)_n$ converge to $\,T(a)$.
But $h(a_n,\bar{z}_n)=h(a_n,\bar{w}_n)=0$, then $\bar{z}_n=\bar{w}_n=q(a_n)$, which is a contradiction. \\ 

Consequently, there exists a negligible closed  set $N$ in $\bar{\Omega}$ such that $\accentset{\circ}{\Omega} \backslash N \subset S$. In addition, $T$ is a $C^1$ function on $\accentset{\circ}{\Omega} \backslash N$. $\qedhere$ 
 \end{proof}
It remains to prove the following\,: 
\begin{proposition} \label{Inv} The matrix $\,F+N$ is invertible for a.e $x\in \Omega.$ 
\end{proposition}  
\begin{proof} It is enough to prove that the determinant of $F+N$ only vanishes at a countable number of points on each transport ray, 
 since it is well-known that a set that meets each transport ray at a countable number of points is negligible and this is due to the fact that the direction of the transport rays is countably Lipschitz (for more details about the proof of this property, we can see for instance Chapter 3 in \cite{8}). To do that, fix $x_0 \in \Omega\,$ and set $x:=T_t(x_0),$ where $t \in [0,1]$.
 Then, we have $$F(x)=F(x_0) \,\;\mbox{and}\;\, N(x)=(1-t)N(x_0).$$
\\
  But $F(x_0)$ is invertible, then we get
  
 \begin{eqnarray*} 
 \mbox{det} (F(x) + N(x)) &=& \mbox{det}\left(F(x_0) + (1-t)N(x_0)\right)\\  \\
 &=& (1-t)^{d-1} \mathrm{det}\left(F(x_0)\right) \mathrm{det}\left(\frac{1}{1-t}I+ [F(x_0)]^{-1} N(x_0)\right),\\
 \end{eqnarray*} 
 which can only vanish at most for a finite number of different values of $\,t$. $\qedhere$
 \end{proof}

\end{document}